\documentclass[12pt,reqno]{amsart}
\usepackage{graphicx}
\usepackage{psfrag}
\usepackage{pxfonts}
\usepackage{mathrsfs}
\usepackage{color}
\usepackage{amsmath,amssymb}

\numberwithin{equation}{section}

\theoremstyle{plain}
\newtheorem{maintheorem}{Theorem}

\newtheorem{theorem}{Theorem}[section]

\newtheorem{cor}[theorem]{Corollary}

\newtheorem{definition}[theorem]{Definition}

\theoremstyle{remark}

\begin{document}
\title[Some sufficient conditions for transitivity of Anosov diffeomorphisms]
{Some sufficient conditions for transitivity of Anosov diffeomorphisms}
\author[F. Micena]{Fernando Micena}
\address{Instituto de Matem\'{a}tica e Computa\c{c}\~{a}o,
  IMC-UNIFEI, Itajub\'{a}-MG, Brazil.}
\email{fpmicena82@unifei.edu.br}


\thanks{I thank Ali Tahzibi, Boris Hasselblatt,  Luis Fernando Mello, Paulo Varandas and Rafael de la Llave  for the valuable conversations, comments and suggestions during the preparation of this paper. I thank the anonymous referee for your comments and suggestions. }
\baselineskip=18pt              


\begin{abstract}
Given a $C^2-$ Anosov diffemorphism $f: M \rightarrow M,$ we prove that the jacobian condition $Jf^n(p) = 1,$ for every point $p$ such that $f^n(p) = p,$ implies transitivity. As application in the celebrated theory of Sinai-Ruelle-Bowen, this result allows us to state a classical theorem of Livsic-Sinai without directly assuming transitivity as a general hypothesis. A special consequence of our result is that every $C^2-$Anosov diffeomorphism, for which every point is regular, is indeed transitive.
\end{abstract}

\subjclass[2010]{}

\maketitle


\section{Introduction}\label{sec:intro}

Let $(M,g)$ be a $C^{\infty}$ compact, connected and boundaryless Riemannian mani-\\fold and $f:M \rightarrow M$  a $C^1-$diffeomorphism. We say that $f$ is an Anosov diffeomorphims if there are numbers $0 < \beta < 1 < \eta, C > 0$ and a $Df-$invariant continuous splitting $T_xM = E^u_f(x) \oplus E^s_f(x),$ such that for any $n \geq 0$ and for all $x \in M$
$$ ||Df^n(x) \cdot v || \geq \frac{1}{C} \eta^n ||v||, \forall v \in E^u_f(x), \quad  ||Df^n(x) \cdot v || \leq C \beta^n ||v||, \forall v \in E^s_f(x).$$

It is known that the bundles $E^s_f, E^u_f$ are  respectively integrable to invariant  stable and unstable foliations denoted $\mathcal{F}^s_f$ and $\mathcal{F}^u_f.$  Given a point $x \in M,$ the stable and unstable leaves that contain $x$ are respectively characterized by
$$W^s_f(x) = \{ y \in M| \; \displaystyle\lim_{n \rightarrow +\infty}d(f^n(x) , f^n(y)) = 0  \},$$
$$W^u_f(x) = \{ y \in M| \; \displaystyle\lim_{n \rightarrow +\infty}d(f^{-n}(x) , f^{-n}(y)) = 0  \}.$$
If $f$ is $C^r, r\geq 1,$ then $W^{\ast}_f(x), \ast \in \{s,u\}$ are embedded $C^r$ submanifolds of $M.$

From now on we denote by $m$ the probability Lebesgue measure on $M$ induced by $g,$ and by $Per(f)$ the set of periodic points of a diffeomorphism $f $ and we denote $Jf(x) = |\det (Df(x): T_xM  \rightarrow T_{f(x)}M)|. $

Anosov diffeomorphims play an important role in the theory of dynamical systems and this class of diffeomorphisms satisfies many rich dynamical pro-\\perties.
An important and  fundamental step toward the classification of Anosov diffeomorphisms would consist in showing that every Anosov diffeomorphism is indeed transitive.

\begin{definition}
Let $(X,d)$ be a compact metric space and $f: X \rightarrow X$ a continuous function. We say that $f$ is transitive if for any nonempty open sets $U$ and $V$ there exists an integer $N \geq 0$ such that $f^{-N}(V) \cap U \neq \emptyset,$ or equivalently, there exists a point $x \in X$ with dense orbit.
\end{definition}

From Theorem C of \cite{Mn}, it is well known that every Anosov diffeomorphism on a infranilmanifold  is transitive.

The present work concerns to provide some sufficient conditions for transitivity of Anosov diffemorphisms. Our main result is the following theorem.

\begin{maintheorem}\label{t3}
Let $f: M \rightarrow M$ be a $C^{2}-$Anosov diffeomorphism. If $Jf^n(p) = 1,$ for any $ p \in Per(f),$ such that $f^n(p) = p,$ then $f$ is
transitive and leaves an invariant $C^1$ volume form.
\end{maintheorem}

The most important part of Theorem \ref{t3} is the proof of the transitivity. The existence of an invariant volume form easily follows from the $C^1$ version of  Livsic's Theorem (Theorem 19.2.5 of \cite{K}).

The above Theorem \ref{t3} allows to rewrite a Theorem of Livsic-Sinai (Theorem 4.14 of \cite{Bowen}) removing transitivity as general hypothesis, as follows.

\begin{cor}
Let f be a $C^2-$Anosov diffeomorphism. The following statements
are equivalent:
\begin{enumerate}
\item $f$ admits an invariant measure of the form $d\nu = hdm$ where $h$ is a positive
$C^1$ function. (See  remark on page 72 of \cite{Bowen}).
\item $f$ admits an invariant measure $\nu$ absolutely continuous w.r.t. $m.$
\item $Jf^n(x) = 1$  whenever $f^n(x) = x.$
\end{enumerate}
\end{cor}

It is important to note that  Theorem 4.14 of \cite{Bowen} is obtained assuming transitivity as general assumption. In fact the items $(1)$ and  $(2)$ imply transitivity, but is not clear if item $(3)$ implies transitivity directly. Our Theorem \ref{t3} answers positively this point. The main tools used to prove it are the theories of SRB measures in \cite{Bowen} and \cite{LEDRAPPIER}.

\begin{definition} Let $f: M \rightarrow M$ be a $C^1$ diffeomorphim. We say that $x \in M$ is a regular point for $f$ if there are real numbers $\lambda_1(x) > \lambda_2(x) > \ldots > \lambda_l(x)$ and a splitting $T_xM = E_1(x) \oplus \ldots \oplus E_l(x)$ of the tangent space of $M$ at $x,$  such that
$$\displaystyle\lim_{n \rightarrow \pm\infty} \frac{1}{n} \log(|| Df^n(x) \cdot v ||) = \lambda_i(x), \forall v \in E_i\setminus\{0\} \; \mbox{and}\; 1\leq i \leq l.$$
The numbers $\lambda_i(x) $ are called Lyapunov exponents of $f$ at $x.$  We denote by $R(f)$ the set of regular points of $f.$
\end{definition}

\begin{theorem}[Oseledec's Theorem] If $f:M \rightarrow M$ is a $C^1-$diffeomorphism preserving $\mu,$ a Borel probability measure, then $\mu(R(f)) = 1.$
\end{theorem}

For the proof of Oseledec's Theorem see \cite{Oseledets} or \cite{Ma}.

Given a point $x \in R(f),$ we denote by $\Lambda^u_f(x)$ the sum of all positive Lyapunov exponent of a point $x.$ More precisely
$$\Lambda^u_f(x) = \sum_{i = 1}^l \max\{\lambda_i(x),0\}\cdot \dim(E_i(x)).$$ Analogously we denote by $\Lambda^s_f(x)$ the sum of all negative Lyapunov exponent of $x \in R(f).$

We can use Theorem \ref{t3} to obtain sufficient conditions for transitivity of Anosov diffeormorphisms.

\begin{cor}\label{t31}
Let $f: M \rightarrow M$ be a $C^{2}-$Anosov diffeomorphism. If $\Lambda^u_f$ and $\Lambda^s_f$ are constant on $Per(f),$ then $\Lambda^u_f + \Lambda^s_f $ is null on $Per(f)$ and consequently  $f$ is transitive.
\end{cor}

\begin{cor}\label{t4}
Let $f: M \rightarrow M$ be a $C^2-$Anosov diffeomorphism. If $R(f) = M, $ then $f$ is transitive.
\end{cor}

Based on \cite{Y} we obtain.

\begin{cor}\label{t5}
 Let $M$ be a compact Riemannian manifold with negative sectional curvature. Then there are no $C^2-$Anosov diffeomorphisms $f: M \rightarrow M,$ such that $R(f) = M.$
\end{cor}

In fact, from \cite{Y}, there are no transitive Anosov diffeomorphisms defined on a compact Riemannian manifold with negative sectional curvature.

\section{Proof of Theorem \ref{t3}}

An important fact of Anosov diffeomorphisms is that such diffeomorphisms are axiom $A.$ Axiom $A$ diffeomorphisms satisfy many rich ergodic and topological properties, see \cite{Bowen}. We have a special interest in SRB measures.

\begin{definition} For a given $C^2-$axiom $A$ diffeomorphism $f: M \rightarrow M,$ a $SRB$ measure for $f$ is an $f-$invariant probability measure $\mu,$ such that  $$h_{\mu}(f) = \displaystyle\int_M \log(J^uf(x)) d\mu,$$ where $J^uf(x) = |\det(Df(x): E^u_f(x) \rightarrow E^u_f(f(x)))|$ and $h_{\mu}(f)$ denotes the metric entropy of $f$ with respect to $\mu.$
\end{definition}

For more definitions and equivalences, see for instance \cite{LS}.

Let $\Omega_s$ be a  basic set of a $C^2-$Axiom A diffeomorphism $f: M \rightarrow M.$ By Theorem 4.11 of \cite{Bowen},  if $m(W^s(\Omega_s)) > 0,$ then there is a measure $\nu$ such that $\nu(\Omega_s) = 1$ and $h_{\nu}(f) = \displaystyle\int_M \log(J^uf(x)) d\nu.$  In other words there is an SRB measure, such that $\nu(\Omega_s) = 1.$

In our case, since $f$ is Anosov, in particular it is axiom A, then
$$ \Omega(f) = \bigcup_{j =1}^n \Omega_j\; \mbox{and} \; M = \displaystyle\bigcup_{j =1}^n W^s(\Omega_j),$$
as in \cite{Bowen}. So there is a basic set $\Omega_s$ such that $m(W^s(\Omega_s)) > 0.$ Let $\nu$ be an SRB measure of $f,$ so

\begin{equation}
 h_{\nu}(f)  = \displaystyle\int_M \log(J^uf(x)) d\nu  =\displaystyle\int_{M}\Lambda^u_f(x)d\nu(x).\label{SRBu}
\end{equation}

Let $\mu$ be an arbitrary $f-$invariant probability measure. Denoting by $R$ the set of all simultaneously regular and recurrent points of $f,$ we get $\mu(R)  = 1.$ By applying Anosov Closing Lemma (Theorem 6.4.15 of \cite{K}),  for any $x \in R$  \begin{equation} \label{shadow} \lim_{n \rightarrow +\infty} \frac{1}{n} \log(Jf^n(x)) = 0, \end{equation}
so $\Lambda^s_f(x) + \Lambda^u_f(x) = 0.$ In fact, consider $\varepsilon > 0$ arbitrary and $x \in R.$ Since $f$ is  $C^1$ take $\delta > 0$ such that $1 - \varepsilon < \frac{Jf(x)}{Jf(y)} < 1 + \varepsilon,$ if $d(x,y) < \delta.$ Using that $x \in R,$ there is a sequence of positive integers $n_k, k =1,2,\ldots$ such that $d(x, f^{n_k}(x)) < \delta',$ for some $\delta' > 0,$  such that every $\delta'-$pseudo orbit is indeed $\delta$ shadowed by a periodic orbit of a point $p_k,$ such that $f^{n_k}(p_k) = p_k$ and $d(f^i(x), f^i(p_k)) < \delta, i = 0, 1, \ldots, n_k-1.$ Since $Jf^{n_k}(p_k) = 1,$
$$Jf^{n_k}(x) = \frac{Jf^{n_k}(x) }{Jf^{n_k}(p_k)} = \frac{\prod_{j=0}^{n_k -1} Jf (f^j(x))}{\prod_{j=0}^{n_k -1} Jf (f^j(p_k))}  \Rightarrow  (1 - \varepsilon)^{n_k} <Jf^{n_k}(x) < (1+ \varepsilon)^{n_k}.$$
Taking $k \rightarrow +\infty, $ we get $ \log(1 - \varepsilon) \leq  \Lambda^u_f(x) + \Lambda^s_f(x)  = \displaystyle\lim_{k \rightarrow +\infty}\frac{1}{n_k} \log(Jf^{n_k} (x)) \leq \log(1 + \varepsilon).$ Finally taking $\varepsilon $ going to zero, we obtain $\Lambda^u_f(x) + \Lambda^s_f(x) = 0,$ for any $x \in R.$

In particular, since $\nu(R) = 1$ we get
\begin{equation}
 h_{\nu}(f) = \displaystyle\int_{M}\Lambda^u_f(x)d\nu(x)= -\displaystyle\int_{M} \Lambda^s_f(x)d\nu(x).\label{SRBs}
\end{equation}

From expressions $(\ref{SRBu}) $  and  $(\ref{SRBs}), $ and using the SRB theory developed in \cite{LEDRAPPIER}, the above expressions mean that $\nu$ has absolutely continuous density along the stable and unstable foliations. Since $f$ is $C^2,$ the measure $\nu$ is absolutely continuous, see  \cite{LEDRAPPIER} as a reference of this fact. As $\nu$ is absolutely continuous, there is a basic set $\Omega_s \subset \Omega(f),$ such that $\nu(\Omega_s) > 0,$ consequently $m(\Omega_s) > 0.$ By Corollary 7.4.7 of \cite{FH}, the set $\Omega_s$ is open, by connectedness $M = \Omega(f).$

For completeness of the proof, we  reproduce the argument of Corollary 7.4.7 of \cite{FH}.

Since  $m(\Omega_s) > 0,$ the basic set $\Omega_s$ is an attractor and then $W^u(\Omega_s) = \Omega_s $ (see Theorem 5.3.27 of \cite{FH}). Consequently $m(W^u(\Omega_s) ) > 0$ and it implies that $f$ is a repeller, so $W^s(\Omega_s) = \Omega_s.$  Finally $W^u(\Omega_s) = \Omega_s = W^s(\Omega_s)$ and then $\Omega_s$ is open. By connectedness of $M,$ we get $M = \Omega_s$ and thus $f$ is transitive.

Since $f$ is a $C^2-$Anosov diffemorphism, $x \in M \mapsto \log(Jf(x))$ is $C^1.$ Now using the $C^1$ version Livsic's Theorem (Theorem 19.2.5 of \cite{K}) there is a $C^1-$function $\phi:M \rightarrow \mathbb{R},$ such that
\begin{equation}\label{density}
\log(J f(x))  =  \phi( f(x))- \phi(x) \Leftrightarrow Jf(x) e^{-\phi(f(x))} = e^{-\phi(x)}.
\end{equation}

Here we observe that if $f$ is $C^{\infty},$ then by \cite{LlaveMM86}, the solution $\phi$ is also $C^{\infty}.$

If we consider the absolutely continuous measure $d\mu_{\phi} = e^{-\phi(x)}dm$, for any Borel set $S,$
$$ \mu_{\phi}(f(S)) =  \displaystyle\int_{f(S)} e^{-\phi(x)}dm = \displaystyle\int_{S} e^{-\phi(f(x))} Jf(x) dm = \displaystyle\int_{S} e^{-\phi(y)} dm = \mu_{\phi}(S), $$
so $\mu_{\phi}$ is $f-$invariant.
Since $e^{\phi(x)} \in [\frac{1}{c}, c]$ for any $x \in M$ and for some $c > 0,$ the measure $\mu_{\phi}$ is finite. Up to normalize $\mu_{\phi}$, we can assume $\mu_{\phi}$ is a probability measure.  Since $f$ is a transitive $C^2-$Anosov diffeomorphism, then by Corollary 4.13 of \cite{Bowen}, there is a unique absolutely continuous invariant probability measure for $f,$ so $\nu = \mu_{\phi}.$

\section{Some Applications}

We present the proofs of Corollaries \ref{t31} and  \ref{t4}.

 \begin{proof}[Proof of Corollary \ref{t31} ] Denote by $\Lambda^u_f$ and $\Lambda^s_f$  the constant values of  $\Lambda^u_f(p) $ and $\Lambda^s_f(p) $ respectively,  for any $p \in Per(f).$
Consider $\mu$  an arbitrary $f-$invariant probability measure and denoting by $R$ the set of all simultaneously regular and recurrent points of $f,$ we get $\mu(R)  = 1.$ We can use  Anosov Closing Lemma to get
\begin{equation}
\Lambda^u_f(x):= \lim_{n \rightarrow + \infty} \frac{1}{n} \log(J^uf^n(x)) = \Lambda^u_f, \label{sameu}
\end{equation}
for any $x \in R.$ The equation $(\ref{sameu})$ is obtained following an analogous argument to get $(\ref{shadow}),$ using the fact $J^uf^{n_k}(p_k) = (\Lambda^u_f)^{n_k}. $

By Ruelle's inequality we obtain $h_{\mu}(f) \leq \Lambda^u_f,$ and using the Variational Principle, the topological entropy of $f,$ denoted by $h_{top}(f),$ satisfies
\begin{equation}\label{topleq}
h_{top}(f) \leq \Lambda^u_f.
\end{equation}

As at the beginning of the proof of Theorem \ref{t3}, there is a measure $\mu^{+}$ an SRB measure of $f.$  So by definition of SRB measure and equation $(\ref{sameu})$, we get
\begin{equation}\label{maxu}
h_{\mu^+}(f) =  \Lambda^u_f,
\end{equation}
so we conclude that $\mu^{+}$ is a  measure of maximal entropy for $f.$

Arguing with $f^{-1},$ we get analogously $h_{top}(f^{-1}) \leq -\Lambda^s_f.$ Taking $\mu^{-}$ an SRB measure of $f^{-1},$  we obtain
\begin{equation}\label{maxs}
h_{\mu^-}(f^{-1}) =  -\Lambda^s_f,
\end{equation}
so we conclude that $\mu^{-}$ is a  measure of maximal entropy for $f^{-1}.$ Since $h_{top}(f) = h_{top}(f^{-1}),$ by  $(\ref{maxu})$ and $(\ref{maxs})$ we get $\Lambda^u_f(p) = -\Lambda^s_f(p),$ for any $p\in Per(f).$ Consequently $Jf^n(p) = 1,$ for any periodic point $p$ such that $f^n(p)=p,$ with $n \geq 1.$ The result follows direct from Theorem \ref{t3}.

\end{proof}

 \begin{proof}[Proof of Corollary \ref{t4} ] The argument here is similar to the Hopf argument.

Let $x_0$ be an arbitrary point on $M,$  since $f$ has local product structure, there is an open neighborhood $V$ of $x_0,$ such that,  given $z \in V,$ there is a point $z'\in V \cap W^u_f(z)\cap W^s_f(x_0).$ Using that every point is regular we obtain
$$ \Lambda^u_f(x_0) = \displaystyle\lim_{n \rightarrow +\infty}\frac{1}{n} \log(J^uf^n(x_0)) =
\displaystyle\lim_{n \rightarrow +\infty}\frac{1}{n} \log(J^uf^n(z')) = \displaystyle\lim_{n \rightarrow -\infty}\frac{1}{n} \log(J^uf^n(z)) =  \Lambda^u_f(z), $$
for any given $z \in V.$

The map $x \mapsto \Lambda^u_f(x)$ is locally constant and $M$ is connect, it implies  $\Lambda^u_f(x) = \Lambda^u_f,$ for any $x \in M.$ Analogously $\Lambda^s_f(x) = \Lambda^s_f,$ for any $x \in M.$  Particularly the functions $\Lambda^u_f$ and $\Lambda^s_f$ are constant on $Per(f).$ We conclude the proof applying Corollary \ref{t31}.
\end{proof}

The existence of Lyapunov exponents everywhere is a rigidity hypothesis,  which was studied in \cite{LM20}.

\end{document}